\documentclass[final]{siamltex}

\usepackage{mathrsfs}
\usepackage{color}
\usepackage{graphicx, epstopdf}
\usepackage{amsmath}
\usepackage{amsfonts}
\usepackage{amssymb}
\usepackage{textcomp}
\usepackage{latexsym}

\usepackage{threeparttable}

\newtheorem{remark}[theorem]{Remark}
\def\cal{\mathcal}

\newcommand\ddelta\bigtriangledown
\newcommand\ld\lambda
\newcommand\Ld\Lambda

\def \bP {\Bbb P}

\def \bZ {\Bbb Z}

\title{Superconvergence of Discontinuous Galerkin method for linear
hyperbolic equations\thanks{The second author was supported in part
by the US National Science Foundation through grant DMS-1115530. The
third author was supported in part by the National Natural Science
Foundation of China under the grant 11171359.}}
\author{Waixiang Cao \footnotemark[2]\ \footnotemark[3]
\and Zhimin Zhang \footnotemark[2]\ \footnotemark[4] \and Qingsong
Zou \footnotemark[3]}

\begin{document}
\maketitle
\renewcommand{\thefootnote}{\fnsymbol{footnote}}
\footnotetext[2] {Beijing Computational Science Research Center,
Beijing, 100084, China.}
 \footnotetext[3]{ College of Mathematics and Computational Science and Guangdong Province Key Laboratory of Computational Science, Sun Yat-sen
          University, Guangzhou, 510275, China. }
          \footnotetext[4]{ Department of Mathematics, Wayne State
University, Detroit, MI 48202, USA.}

%------------------------------------------------------------------------------------------------------

\begin{abstract}
    In this paper, we study superconvergence properties of the
    discontinuous Galerkin (DG) method for one-dimensional linear
    hyperbolic equation when upwind fluxes are used.
    We prove, for any polynomial degree $k$,
    the $2k+1$th (or $2k+1/2$th) superconvergence rate of the DG approximation at the
   downwind points  and for the  domain average under
   quasi-uniform meshes and some suitable initial discretization.
   Moreover, we prove that the derivative approximation of  the DG solution is
    superconvergent  with a rate $k+1$ at all interior left Radau points.
    All theoretical finding are confirmed by numerical experiments.
\end{abstract}

\begin{keywords}
   Discontinuous Galerkin method, superconvergence, hyperbolic, Radau
points, cell average, initial discretization
\end{keywords}

\begin{AMS}
 65M15, 65M60, 65N30
\end{AMS}

\section{ Introduction}

  Discontinuous Galerkin (DG) methods, originally developed for
  neutron transport problems \cite{Reed;Hill:1973}, are a class of finite element methods
  using completely discontinuous piecewise polynomial space. Due to its
  flexibility for arbitrarily unstructured meshes, the efficiency in parallel implementation,
  the ability to easily handle complex geometries or interfaces and  accommodate arbitrary $h$-$p$
  adaptivity, the DG method gains more popularity in solving various
  differential equations and attracts intensive theoretical
  studies. We refer to
  \cite{Cockburn;Hou;Shu:1990,Cockburn;Lin;Shu:1989,Cockburn;Shu:1989,Cockburn;Shu:1991,Cockburn;Shu:JCP1998,Cockburn;Shu:SIAM1998}
   and the references cited therein for
  an incomplete list of references.

 In the past several decades, there also has been considerable
 interest in studying superconvergence properties of DG methods.
  We refer to
\cite{Adjerid;Devine2002,Adjerid;Massey2006,Xie;Zhang2012,Zhang;xie;zhang2009}
for ordinary differential equations, and
  \cite{Adjerid;Weinhart2009,Adjerid;Weinhart2011}
  for multidimensional first order hyperbolic systems, and \cite{Chen;Shu:JCP2008,Chen;Shu:SIAM2010}
  for one-dimensional hyperbolic conservation laws and
  time-dependent convection-diffusion equations.
  Very recently, Yang and Shu in \cite{Yang;Shu:SIAM2012}  studied
   superconvergence properties of a DG method for linear hyperbolic equations
   when upwind fluxes were used. They proved a
  $k+2$th superconvergence rate of the DG approximation at the right Radau points
   and for the cell average
  under suitable initial discretization. They also
  presented  numerically, for $k=1,2$, a $2k+1$th superconvergence rate of the DG solution at the
  downwind points  and for the cell average. However,
   a theoretical proof of this remarkable property remains open.
  Indeed, the $2k+1$th superconvergence rate is one of the unsolved mysteries of the DG method  for hyperbolic
  equations.

   The main purpose of our current work is to uncover this mystery by offering a rigorous mathematical proof
  for the $2k+1$th (or $2k+1/2$th) superconvergence rate at downwind points and for the domain average.
  As by-products, we provide a simplified proof for  the point-wise $k+2$th superconvergence rate at the right Radau points,
  a fact established in \cite{Yang;Shu:SIAM2012} in a weaker sense (under a discrete $L^2$-norm) by a different approach;
  we also prove  a point-wise $k+1$th derivative superconvergence rate at the left Radau points, a fact not established before.
  By doing so, we present a full picture for superconvergence properties of the DG method for liner hyperbolic equations
  in one spacial dimension.

     To prove the $2k+1$th superconvergence rate, we revisit the problem
     considered in \cite{Yang;Shu:SIAM2012} and make the same assumption that the time integration is exact.
     The novelty lies in that we adopt a completely different analysis track.
    An essential ingredient is the design of a correction function $w$. The idea is
    motivated from its successful applications to finite element methods (FEM) and finite volume
   methods (FVM) for elliptic equations (see, e.g. \cite{Cao;zhang;zou:2kFVM,Chen.C.M2012}).
   However, as the correction function is very different from FVM to FEM,
   it is much more so for the DG method due to special features of hyperbolic equations from those of elliptic equations,
   especially, the time dependent feature.
   Our approach here is to correct the error between the exact solution $u$ and its truncated Radau expansion
     $P_h^-u$ (defined in Section 3), which interpolates $u$ at all downwind points. With help of the correction function $w$,
     which is zero at all downwind points, we prove that the DG solution is superclose (with order $2k+1$) to $P_h^-u-w$.
     It is this supercloseness that leads to the $2k+1$th superconvergence rate at the downwind points (in average sense)
     and for the domain average.
     As a direct consequence, we obtain another new theoretical result : the derivative approximation of the DG solution is
    superconvergent at all interior left Radau points with a rate $k+1$.
    To end this introduction, we would like to point out that all superconvergent results here are valid
    for one-dimensional linear systems, and the proof is along the same line without any difficulty.

    The rest of the paper is organized as follows. In Section 2, we present DG schemes
    for linear conservation laws.
    Section 3 is the most technical part, where we
    construct a special interpolation function superclose to the DG
    solution.
   Section 4 is the main body of the paper, where superconvergence
   results are proved with suitable initial discretization.
    Finally, we provide some
   numerical examples to support our theoretical findings in Section 5.

   Throughout this paper,  we adopt standard notations for Sobolev spaces such as $W^{m,p}(D)$ on sub-domain $D\subset\Omega$ equipped with
    the norm $\|\cdot\|_{m,p,D}$ and semi-norm $|\cdot|_{m,p,D}$. When $D=\Omega$, we omit the index $D$; and if $p=2$, we set
   $W^{m,p}(D)=H^m(D)$,
   $\|\cdot\|_{m,p,D}=\|\cdot\|_{m,D}$, and $|\cdot|_{m,p,D}=|\cdot|_{m,D}$. Notation``$A\lesssim B$" implies that $A$ can be
  bounded by $B$ multiplied by a constant independent of the mesh size $h$.
  ``$A\sim B$" stands for $``A\lesssim B"$ and $``B\lesssim A"$.

\section{ DG schemes}
  We consider the discontinuous Galerkin method for the following one-dimensional linear hyperbolic conservation
 laws
\begin{eqnarray}\label{con_laws}
\begin{aligned}
   &u_t+u_x=0,\ \ &&(x,t)\in [0,2\pi]\times(0,T], \\
   &u(x,0)=u_0(x),\ \  &&x\in R,
\end{aligned}
\end{eqnarray}
  where $u_0$ is sufficiently smooth. We will consider both the periodic boundary
  condition $u(0,t)=u(2\pi,t)$ and the Dirichlet boundary
   condition $u(0,t)=g(t)$.

  Let $\Omega=[0,2\pi]$ and $0=x_{\frac 12}<x_{\frac
  32}<\ldots<x_{N+\frac 12}$ be $N+1$ distinct points on the
  interval $\bar{\Omega}$. For all positive integers $r$, we define
  $\bZ_{r}=\{1,\ldots,r\}$ and denote by
\[
    \tau_j=(x_{j-\frac 12},x_{j+\frac 12}),\ \ x_j=\frac 12(x_{j-\frac 12}+x_{j+\frac
    12}),\ j\in\bZ_N
\]
  the cells and cell centers, respectively.
  Let $h_j=x_{j+\frac 12}-x_{j-\frac 12}$,  $\bar{h}_j = h_j/2$
  and $h =  \displaystyle\max_j\; h_j$. We assume  that the mesh
  is quasi-uniform, i.e., there exists a constant $c$ such that
\[
    h\le c h_j, j\in\bZ_N.
\]

  Define
\[
    V_h=\{ v: \; v|_{\tau_j}\in P_k(\tau_j),\; j\in\bZ_N\}
^{}\]
  to be the finite element space, where $P_k$ denotes the space of
  polynomials of degree at most $k$  with coefficients as functions of $t$.
  The DG scheme for \eqref{con_laws} reads as: Find $u_h\in V_h$
  such that for any $v\in  V_h$
\begin{equation}\label{DG_scheme1}
   (u_{ht},v)_j-(u_h,v_x)_j+u_h^-v^-|_{j+\frac
    12}-u_h^-v^+|_{j-\frac12}=0,
\end{equation}
  or equivalently,
\[
    (u_{ht}+u_{hx},v)_j+[u_h]_{j-\frac
    12}v^+_{j-\frac12}=0.
\]
   Here  $(u_{ht},v)_j=\int_{\tau_j}u_{ht}v dx$,
   $v^-_{j+\frac 12}$ and $v^+_{j+\frac12}$  denote the left
  and right limits of $v$ at the point $x_{j+\frac
    12}$, respectively, and $[v]_{j-\frac 12}=v^+_{j-\frac12}-v^-_{j-\frac 12}$
    denotes the jump of $v$ across $x_{j-\frac 12}$.

  Define
\[
   H_h^1=\{v: \; v|_{\tau_j}\in H^1(\tau_j),\; j\in\bZ_N\}
\]
  and for all $w,v\in H_h^1$,  let the bilinear form
\[
    a(w,v)=\sum_{j=1}^{N}a_j(w,v),
\]
  where
\[
  a_j(w,v)=(w_t,v)_j-(w,v_x)_j+w^{-}v^-|_{j+\frac 12}-w^{-}v^+|_{j-\frac
  12}.
\]
 With this notation, the DG scheme \eqref{DG_scheme1} can be rewritten as
\[
a(u_h, v)=0, \quad v\in V_h.
\]
Obviously, the exact solution $u$  also satisfies
\[
   a(u,v)=0, \quad v\in V_h.
\]
Moreover, if $v\in H_h^1$ satisfies $v^-_{\frac 12}=0$ or
$v^-_{\frac
    12}=v^-_{N+\frac 12}$, then
\begin{eqnarray*}
 a(v,v)&=&(v_t,v)+\frac 12\sum_{j=1}^{N}\left(v^-_{j+\frac 12}v^-_{j+\frac 12}+v^+_{j-\frac 12}v^+_{j-\frac 12}-2v^-_{j-\frac 12}v^+_{j-\frac
 12}\right)\\
 &=&(v_t,v)+\frac 12\sum_{j=1}^{N}[v]^2_{j-\frac 12}+\frac 12 (v^-_{N+\frac 12}v^-_{N+\frac 12}-v^-_{\frac 12}v^-_{\frac 12})\\
 &\ge&(v_t,v),
\end{eqnarray*}
 which means that in both two cases,
\begin{equation}\label{ineq1}
\frac 12\frac{d}{dt}\|v\|_0^2=(v_t,v)\le a(v,v).
\end{equation}

%%%%%%%%%%%%%%%%%%%%%%%%%%%%%%%%%%%%%%%%%%%%%%%%%%%%%%%%%%%%%%%%%%%%%%%%%%%%%%%%%%%%%%%%%%

\section{ Construction of a special interpolation function}
\setcounter{equation}{0}

  One of the superconvergence analysis methods in FEM is through estimating
\[
     a(u-u_I,v),\ \ \forall v\in V_h,
\]
  where $u_I\in V_h$ is  a specially designed interpolation function, which
  is superclose to $u_h$ such that
\[
    a(u-u_I,v)=a(u_h-u_I,v),\ \ \forall v\in V_h
\]
  is of high order.  Our analysis here is also along this
  line.

  We begin the construction of $u_I$ with the Gauss-Radau projection
$P_h^-u\in V_h$ of $u$ defined by
\begin{equation}\label{gr}
(P^-_hu,v)_j=(u,v)_j,\forall
v\in\bP^{k-1}(\tau_j)\quad\text{and}\quad
P^-_hu(x_{j+\frac12}^-)=u(x_{j+\frac12}^-).
\end{equation}
   Notice that this special projection is used in the error estimates
of the DG methods, e.g. in
\cite{Chen;Shu:SIAM2010,Yang;Shu:SIAM2012}.  Since  in each element
$\tau_j, j\in\bZ_N$,   $u(x,t)$ has the
  following  Radau expansion
\begin{equation}\label{expansion}
   u(x,t)=u(x^-_{j+\frac 12},t)+\sum_{m=1}^\infty
   u_{j,m}(t)(L_{j,m}-L_{j,m-1})(x),
\end{equation}
where $L_{j,m}, j\in \bZ_N, m\ge 1$ is the classic Legendre
polynomial of degree $m$ in the interval $\tau_j$  and the
coefficient
\begin{equation}\label{coefficient}
u_{j,m}(t)=u(x^-_{j+\frac 12},t)-\frac{1}{h_j}\int_{\tau_j}
u(x,t)\sum_{l=0}^{m-1}(2l+1)L_{j,l}(x)dx,
\end{equation}
the projection $P^-_h u$ has the representation
\begin{equation}\label{eq3}
   (P^-_hu)(x,t)=u(x^-_{j+\frac 12},t)+\sum_{m=1}^k
   u_{j,m}(t)(L_{j,m}-L_{j,m-1})(x).
\end{equation}

A direct calculation yields that  for all $v\in
    V_h$
\begin{equation}\label{equa1}
    a_j(u-P_h^-u,v)=(u_t-P_h^-u_t,v)_j=-u'_{j,k+1}(t)(L_{j,k},v)_{j},
\end{equation}
   where $u_{j,k+1}$ is the same as in \eqref{expansion}.   Since $u_{j,k+1}$ is only of order $h^{k+1}$,
   $a_j(u-P_h^-u,v)$ is also of  ${\cal O}(h^{k+1})$.
In the following, we will find a suitable function $w^l\in V_h$ such
that $a_j(u-P_h^-u+w^l,v)$ has higher order   ${\cal O}(h^{k+l+1})$
for $1\le l\le k$.

For all $v\in H_h^1$, let its average primal function in $\tau_j$ be
\begin{equation}\label{eq7}
     D^{-1}_sv(x)=\frac{1}{\bar{h}_j}\int_{x_{j-\frac 12}}^xv(x') dx'=\int_{-1}^s\hat{v}(s')ds', \quad x\in \tau_j.
\end{equation}
  where
\[
    s=(x-x_j)/\bar{h}_j\in[-1,1],\ \ \hat{v}(s)=v(x).
\]
In each element $\tau_j,j\in\bZ_N, i\ge 0$,  we define
\begin{equation}\label{correcti}
   F_1=P_h^-D^{-1}_sL_{j,k},\ \  F_i=-P_h^-D^{-1}_sF_{i-1}=-(-P_h^-D^{-1}_s)^{i}
   L_{j,k},\ i\ge 2.
 \end{equation}
%\medskip
%The defined functions $F_i, 1\le i\le k,$ have the following properties.
\begin{lemma}For all $1\le i\le k$, $F_i$ has the representation
\begin{equation}\label{fir}
 F_i(x)=\sum_{m=k-i+1}^kb_{i,m}(L_{j,m}-L_{j,m-1})(x), x\in\tau_j,
\end{equation}
  where the coefficients $b_{i,p}$ are independent of the mesh size $h_j$. Consequently,
\begin{eqnarray}\label{fp1}
& F_i(x_{j+\frac12}^-)=0, \ \ \ \|F_i\|_{0,\infty,\tau_j}\lesssim 1,\\
\label{fp2}
 & (F_i,v)_j=0,\quad \forall v\in \bP_{k-i-1}(\tau_j). %\\\label{fp3}
\end{eqnarray}
\end{lemma}
\begin{proof}
We will show \eqref{fir} by induction. First, a straightforward
calculation yields
\[
  D_s^{-1}L_{j,k}=\frac{1}{2k+1}(L_{j,k+1}-L_{j,k-1}),
\]
and thus
\begin{equation*}\label{F1z}
   F_1=\frac{1}{2k+1}(L_{j,k}-L_{j,k-1}),
\end{equation*}
which implies \eqref{fir} is valid for $i=1$ with
$b_{1,k}=\frac{1}{2k+1}.$

 Now we suppose $F_i, i\le k-1$ has the representation \eqref{fir}. Since for all $m\ge 1$,
 \[
 D^{-1}_sL_{j,m}=\frac{1}{2m+1}(L_{j,m+1}-L_{j,m-1})
 \]
 and
 \[
 P_h^-L_{j,k+1}=L_{j,k},  \quad P_h^-L_{j,m}=L_{j,m}, \quad m\le k,
 \]
it is easy to deduce that
\[
   F_{i+1}=-P_h^{-}D_s^{-1}F_i=\sum_{m=k-i}^{k}b_{i+1,m}(L_{j,m}-L_{j,m-1}),
\]
 with
\begin{eqnarray*}
    &&b_{i+1, k-i}=\frac{b_{i,k-i+1}}{2(k-i)+1},\\
    &&b_{i+1, k-i+1}=\frac{b_{i,k-i+2}}{2(k-i)+3}-\frac{b_{i,k-i+1}}{2(k-i)+3}+\frac{b_{i,k-i+1}}{2(k-i)+1},\\
    &&b_{i+1, k}=-\frac{b_{i,k-1}}{2k-1}-\frac{b_{i,k}}{2k+1}+\frac{b_{i,k}}{2k-1},
\end{eqnarray*}
and for all $m=k-i+2,\ldots,k-1$,
\[
    b_{i+1,m}=\frac{b_{i,m+1}-b_{i,m}}{2m+1}+\frac{b_{i,m}-b_{i,m-1}}{2m-1}.
\]
Therefore \eqref{fir} is valid for $i+1$. Consequently, \eqref{fir}
is valid for all $1\le i\le k$.

Since $L_{j,m}(x_{j+\frac12}^-)=1$ for all $m\ge 1$, the first
formula of \eqref{fp1} holds. Moreover, by the iterative relations
between the coefficients of $b_{i,m}, 1\le i\le k, k-i+1\le m\le k$,
we have $|b_{i,m}|\lesssim 1$,  the second of \eqref{fp1} follows
from the fact that $\|L_{j,m}\|_{0,\infty,\tau_j}=1$. Finally, by
the orthogonality of the Legendre polynomials, the  formula
\eqref{fp2} is valid.
\end{proof}

 We are now ready to construct our correction function
for all $1\le l\le k$.  We define, at the boundary point $x=x_{\frac
12}=0$,
\[
    w^l(x_{\frac 12}^{-},t)=0,\ \ \forall t\ge 0;
\]
  and in each element $\tau_j, j\in\bZ_N$,
\begin{equation}\label{correct}
    w^l(x,t)=\sum_{i=1}^l w_i(x,t),\ \ w_i(x,t)=\bar{h}^i_jG_i(t)F_i(x)
\end{equation}
  with
\begin{equation}\label{correct1}
 G_i(t)=u^{(i)}_{j,k+1}(t).
\end{equation}
By  the first formula in \eqref{fp1}, $F_i(x^-_{j+\frac 12})=0$,
then $w_i(x^-_{j+\frac 12},t)=0,$ for all $i\ge 0$. Then for all
$1\le l\le k$,
\begin{equation}\label{corr1}
   w^l(x^-_{j+\frac 12},t)=0, \ \ \forall t,   j\in\bZ_N.
\end{equation}

In the following, we define the special interpolation function
\begin{equation}\label{interpolation}
  u^l_I=P_h^-u-w^l
\end{equation}
and discuss the properties of $a_j(u- u^l_I,v)$.

 \begin{theorem}\label{theorem1}
 Let $u_I^l\in V_h$ be defined by \eqref{interpolation}, \eqref{correct},\eqref{correct1} and \eqref{correcti} with some $1\le l\le k$. Then
   if $u\in W^{k+l+2,\infty}(\Omega),k\ge 1$, we have
\begin{equation}\label{corr2}
    |a_j(u-u_I^l,v)|\lesssim h^{k+l+1}\|u\|_{k+l+2,\infty,\tau_j}\|v\|_{0,1,\tau_j},\ \
    \forall v\in V_h.
\end{equation}
\end{theorem}
\begin{proof}
By the definition of $a_j(\cdot,\cdot)$ and the fact that
$w_i(x^-_{j-\frac 12})=0$,
\[
  a_j(w_i,v)=(w_{it},v)_j-(w_i,v_x)_j.
 \]
By the definition of $w_i$, we have
\begin{eqnarray*}\label{eq55}
\begin{aligned}
 (w_{it},v)_j & =\bar{h}_j^{i}G_{i+1}(F_i,v)_{j}\\
    & =-\bar{h}_j^{i+1}G_{i+1}(D_s^{-1}F_i,v_x)_{j}\\
   &=\bar{h}_j^{i+1}G_{i+1}(F_{i+1},v_x)_{j}\\
   &=(w_{i+1},v_x)_{j},
 \end{aligned}
\end{eqnarray*}
  where in the second equality, we have used the integration by parts
and the fact that
$D_s^{-1}F_i(x_{j+\frac12}^-)=D_s^{-1}F_i(x_{j-\frac12}^+)=0$.
 Then \[a_j(w_i,v)=(w_{i+1}-w_i,v_x)_j\] and thus
\[
a_j(w^l,v)=\sum_{i=1}^la_j(w_i,v)=(w_{l+1}-w_1,v_x)_{j}.
\]
By \eqref{equa1}, $a_j(u-P_h^-u,v)=(w_1,v_x)_j$. Then
\begin{equation}\label{correction}
  a_j(u-u_I^l,v)=a_j(u-P_h^-u+w^l,v)=(w_{l+1},v_x)_{j}=\bar{h}_j^{l}G_{l+1}(F_l,v)_{j}.
\end{equation}
Consequently,
\begin{eqnarray*}
   |a_j(u-u^l_I,v)|&\lesssim&
   h_j^l\|v\|_{0,1,\tau_j}|G_{l+1}|\\
   &\lesssim&
   h_j^{k+l+1}\|u\|_{k+l+2,\infty,\tau_j}\|v\|_{0,1,\tau_j},
\end{eqnarray*}
 where in the last inequality, we have used the fact that
\[
   |G_{l+1}|=|D_t^{l+1}u_{j,k+1}|\lesssim
   h_j^{k+1}\|\partial_x^{k+1}\partial_t^{l+1}u\|_{0,\infty,\tau_j}.
\]
   The proof is completed.
\end{proof}

%\medskip
%\begin{remark}
%The corrected function $u_I=P^-_hu-w$ is the interpolation function we are looking for.
%\end{remark}

\begin{remark}
 As a direct consequence of \eqref{corr2},
\begin{equation}\label{corre3}
    \left|a(u-u^l_I,v)\right| \lesssim
   h^{k+l+1}\|u\|_{k+l+2,\infty}\|v\|_{0,1}.
\end{equation}
\end{remark}

\section{ Superconvergence}

   In this section, we shall study superconvergence properties of DG
   solution at some special points : downwind points and Radau
   points,  and for the domain
   average.

  We begin with a study of the difference between the interpolation function $u_I^l= P^-_hu-w^l$ and the DG solution
   $u_h$.
\begin{theorem}\label{theo:0}
    Let $u\in W^{k+l+2,\infty}(\Omega),u_h\in V_h$  be the solution of
   \eqref{con_laws} and  \eqref{DG_scheme1} respectively.
   Suppose $u_I^l\in V_h$ is  defined by \eqref{interpolation}, \eqref{correct},\eqref{correct1} and
  \eqref{correcti}.
   Then for both the Dirichlet  and  periodic boundary
  condition,
   \begin{equation}\label{spclossness}
  \|u_I^l-u_h\|_0(t)
   \lesssim \|u_I^l-u_h\|_0(0)+th^{k+l+1}\|u\|_{k+l+2,\infty}.
  \end{equation}
\end{theorem}
\begin{proof}Since
\[
   (u_I^l-u_h)^-_{N+\frac 12}=(u_I^l-u_h)^-_{\frac 12}
\]
  for the periodic boundary
  condition and
\[
   (u_I^l-u_h)^-_{\frac 12}=0
\]
   for the Dirichlet boundary condition, \eqref{ineq1} is valid for both two cases if we choose  $v=u_I^l-u_h$.
   Noticing \eqref{corre3},
   we have
\begin{eqnarray*}
  \|u_I^l-u_h\|_0 \frac {d}{dt}\|u_I^l-u_h\|_0&\le&
   \left|a(u_h-u_I^l,u_I^l-u_h)\right|\\
   &=&\left|a(u-u_I^l,u_I^l-u_h)\right|\\
  &\lesssim& h^{k+l+1}\|u\|_{k+l+2,\infty}\|u^l_I-u_h\|_0.
\end{eqnarray*}
 Then
\begin{equation}\label{iii1}
 \frac{d}{dt}\|u_I^l-u_h\|_0
   \lesssim h^{k+l+1}\|u\|_{k+l+2,\infty}
\end{equation}
and \eqref{spclossness} follows.
\end{proof}

\begin{remark}
    From  Theorem \ref{theo:0}, we know that
the suitable choice of the initial solution is of great importance.
To guarantee the superconvergence  rate $k+l+1$ for
$\|u_I^l-u_h\|_0$, the initial error should reach the same
convergence rate,  that is
\begin{equation}\label{initial0}
 \|u_h(\cdot,0)-u_I^l(\cdot,0)\|_0\lesssim h^{k+l+1}\|u\|_{k+l+2,\infty}.
\end{equation}
 We shall demonstrate this point in our numerical
analysis. To obtain \eqref{initial0}, a natural way of initial
discretization is to choose
\begin{equation}\label{initial}
   u_h(x,0)=u_I^l(x,0), \forall x\in\Omega.
\end{equation}
\end{remark}

\subsection{Superconvergence  at the downwind points}

We are now ready to present our superconvergence results of DG
solution at the downwind points.
\begin{theorem}\label{theo:1}
    Let $u\in W^{2k+2,\infty}(\Omega)$  be the solution of
   \eqref{con_laws}, and $u_h$  the solution of \eqref{DG_scheme1}
    with initial value  $u_h(\cdot,0)$ be chosen such that \eqref{initial0} holds with
   $l=k$. Then for both the Dirichlet  and  periodic boundary
  condition,
\begin{equation}\label{super_node2}
   |(u-u_h)(x^-_{j+\frac 12},t)|\lesssim (1+t)h^{2k+\frac
   12}\|u\|_{2k+2,\infty},\ \ \forall j\in\bZ_N,
\end{equation}
  and
\begin{equation}\label{averagenode}
   \left(\frac 1N\sum_{j=1}^N\big(u-u_h\big)^2\big(x^-_{j+\frac 12},t\big)\right)^{\frac 12}\lesssim
   (1+t)h^{2k+1}\|u\|_{2k+2,\infty}.
\end{equation}
  Moreover, if we choose the initial value
  $u_h(\cdot,0)=u_I^k(\cdot,0)$, we have the following improved results
\begin{equation}\label{super_node3}
   |(u-u_h)(x^-_{j+\frac 12},t)|\lesssim th^{2k+\frac
   12}\|u\|_{2k+2,\infty},\ \ \forall j\in\bZ_N,
\end{equation}
and
\begin{equation}\label{averagenode1}
   \left(\frac 1N\sum_{j=1}^N\big(u-u_h\big)^2\big(x^-_{j+\frac 12},t\big)\right)^{\frac 12}\lesssim
   th^{2k+1}\|u\|_{2k+2,\infty}.
\end{equation}
\end{theorem}
\begin{proof} If $u_h(\cdot,0)$ satisfies  \eqref{initial0} with $l=k$, by \eqref{iii1},
\begin{equation}\label{l2estimate}
   \|u_I-u_h\|_0(t)=\|u_I-u_h\|_0(0)+\int_0^t\frac
   {d}{dt}\|u_I-u_h\|_0(s) ds\lesssim (1+t)h^{2k+1}\|u\|_{2k+2,\infty},
\end{equation}
where $u_I=u_I^k$.
  For any fixed $t$, $u_I-u_h\in \bP_k$ in each $\tau_j,j\in \bZ_N$. Then the inverse inequality holds and thus,
\begin{eqnarray*}
   \left|(u_I-u_h)(x^-_{j+\frac
   12},t)\right|&\le &\|u_I-u_h\|_{0,\infty,\tau_j}(t)\\
    &\lesssim &h^{-\frac
  12}\|u_I-u_h\|_{0,\tau_j}(t)\\
  &\lesssim&(1+t)h^{2k+\frac
   12}\|u\|_{2k+2,\infty}.
\end{eqnarray*}
By the fact that $u(x_{j+\frac12}^-)=P_h^-u(x_{j+\frac12}^-)$ and
$w^k(x_{j+\frac12}^-)=0$, we have
\[
u_I(x_{j+\frac12}^-)=u(x_{j+\frac12}^-),\quad \forall j\in\bZ_N.
\]
Then the desired result \eqref{super_node2} follows.

We next show \eqref{averagenode}. Again by the inverse inequality,
\begin{eqnarray*}
\sum_{j=1}^N\|u_I-u_h\|^2_{0,\infty,\tau_j}\lesssim \sum_{j=1}^N
h_j^{-1}\|u_I-u_h\|^2_{0,\tau_j}
 \lesssim N\|u_I-u_h\|_0^2.
\end{eqnarray*}
 Then
\begin{eqnarray*}
    \frac 1N\sum_{j=1}^N\big(u-u_h\big)^2\big(x^-_{j+\frac
    12},t\big)&=&\frac 1N\sum_{j=1}^N\big(u_I-u_h\big)^2\big(x^-_{j+\frac
    12},t\big)\\
    &\le &\frac1N\sum_{j=1}^N\|u_I-u_h\|^2_{0,\infty,\tau_j}(t)\\
    &\lesssim& \|u_I-u_h\|_0^2(t).
\end{eqnarray*}
  The inequality \eqref{averagenode} follows directly from the estimate \eqref{l2estimate}.

  If the initial value $u_h(\cdot,0)=u_I(\cdot,0)$, then
\[
   \|u_I-u_h\|_0(t)=\int_0^t\frac
   {d}{dt}\|u_I-u_h\|_0(s) ds\lesssim
   th^{2k+1}\|u\|_{2k+2,\infty}.
\]
  Following the same line, we obtain \eqref{super_node3} and
  \eqref{averagenode1} directly.
\end{proof}

\begin{remark}
   By \eqref{l2estimate}, the interpolation function $u_I$ is superclose
   to the DG solution $u_h$, with the superconvergence rate $2k+1$.
\end{remark}

\subsection{ Superconvegence for the domain average}

  We have the following superconvergence results for the domain average of $u-u_h$.

\begin{theorem}\label{coro:1}
    Let $u\in W^{2k+2,\infty}(\Omega)$  be the solution of
   \eqref{con_laws}, and $u_h$  the solution of \eqref{DG_scheme1}.
   Suppose the initial solution $u_h(\cdot,0)=P_h^{-}u(\cdot,0)-w^k(\cdot,0)$ with $w^k$ defined
   by \eqref{correct}. Then
\begin{equation}\label{cell-average}
   \left| \frac{1}{2\pi}\int_{0}^{2\pi}(u-u_h)(x,t)dx\right|\lesssim
   (h^{1\over 2}+t^2)h^{2k+\frac 12}\|u\|_{2k+2,\infty}
\end{equation}
   for the Dirichlet boundary condition and
\begin{equation}\label{cell-average1}
   \left| \frac{1}{2\pi}\int_{0}^{2\pi}(u-u_h)(x,t)dx\right|\lesssim
     h^{2k+1}\|u\|_{2k+1,\infty}
\end{equation}
  for the periodic boundary condition.
\end{theorem}
\begin{proof}
    We first estimate the domain average of $u-u_h$ at time $t=0$.
   Note that
\[
   \int_{0}^{2\pi}(u-u_h)(x,0)dx=\int_{0}^{2\pi}(P_h^-u-u_h)(x,0)dx=\int_{0}^{2\pi}w^k(x,0)dx
\]
   By \eqref{fir}, \eqref{correct}-\eqref{correct1}, we derive
\[
    \int_{x_{j-\frac 12}}^{x_{j+\frac 12}} w^k(x,t) dx=\int_{x_{j-\frac 12}}^{x_{j+\frac 12}} w_k(x,t)
    dx=\bar{h}_j^kG_{k}\int_{x_{j-\frac 12}}^{x_{j+\frac 12}} F_k(x)
    dx,\ \ \forall j\in\bZ_N.
\]
   Here $G_{k}$ and $F_k$ are the same as in \eqref{correct}.  By the approximation theory, we have
\[
    |G_k|\lesssim
    h_j^{k+1}\|\partial_t^{k}\partial_x^{k+1}u\|_{0,\infty,\tau_j}\lesssim h_j^{k+1}\|u\|_{2k+1,\infty,\tau_j}.
\]
Then
\[
   \left|\int_{x_{j-\frac 12}}^{x_{j+\frac 12}} w^k(x,t)
   dx\right|\lesssim h^{2k+2}\|u\|_{2k+1,\infty,\tau_j},
\]
  which yields
\[
  \left|\int_{0}^{2\pi} w^k(x,t)
   dx\right|= \left|\sum_{j=1}^N \int_{x_{j-\frac 12}}^{x_{j+\frac 12}} w^k(x,t)
   dx\right|\lesssim h^{2k+1}\|u\|_{2k+1,\infty}.
\]
  Thus,
\[
   \left|\int_{0}^{2\pi}(u-u_h)(x,0)dx\right|\lesssim h^{2k+1}\|u\|_{2k+1,\infty}.
\]
   On the other hand, taking $v=1$ in \eqref{DG_scheme1} and summing
   up for all $j$, we obtain
\begin{eqnarray*}
    \int_{0}^{2\pi}(u-u_h)_t(x,t)dx &=&\sum_{j=1}^N (u-u_h)^-_{j+\frac
    12}(t)-(u-u_h)^-_{j-\frac 12}(t)\\
    &=&(u-u_h)^-_{N+\frac
    12}(t)-(u-u_h)^-_{\frac 12}(t).
\end{eqnarray*}
  Then we have, for the periodic boundary condition,
\[
   \frac{d}{dt}
   \int_{0}^{2\pi}(u-u_h)(x,t)dx=\int_{0}^{2\pi}(u-u_h)_t(x,t)dx=0,
\]
  and for the Dirichlet boundary condition,
\[
   \left|\frac{d}{dt}
   \int_{0}^{2\pi}(u-u_h)(x,t)dx\right|=\left|(u-u_h)^-_{N+\frac
    12}(t)\right|\lesssim th^{2k+\frac 12}\|u\|_{2k+2,\infty},
\]
  where in the last step, we have used
 \eqref{super_node3}.
   Note that
\[
     \int_{0}^{2\pi}(u-u_h)(x,t)dx= \int_{0}^{2\pi}(u-u_h)(x,0)dx+\int_0^t\frac{d}{dt}
   \int_{0}^{2\pi}(u-u_h)(x,t)dxdt.
\]
  Then the desired results follow.
\end{proof}

\subsection{Superconvergence  at the
 left Radau points}
  We denote by $R_{j,l}, l=0,\ldots,k$
  the left Radau points on the interval $\tau_j,j\in\bZ_N$,
  namely, the zeros of $L_{j,k+1}+L_{j,k}, j\in\bZ_N$. We shall prove that the derivative
  error of $u-u_h$ is superconvergent at all  left Radau
  points $R_{j,l},l\in\bZ_k$ except the point $R_{j,0}=x_{j-\frac 12}$.

\begin{lemma}
    Let $u\in W^{k+2,\infty}(\Omega)$  be the solution of
   \eqref{con_laws}. Then
\begin{equation}\label{radau_1}
   \left|\frac{\partial(u-P_h^-u)}{\partial x}(R_{j,l},t)\right|\lesssim h^{k+1}\|u\|_{k+2,\infty}, \ \
   \forall t>0, j\in\bZ_N, l\in\bZ_k.
\end{equation}
\end{lemma}
\begin{proof}
   In each element $\tau_j, j\in\bZ_N$,  we have, from \eqref{expansion} and \eqref{eq3},
\[
    \frac{\partial(u-P_h^-u)}{\partial x}(x,t)=\sum_{m=k+1}^{\infty}u_{j,m}(t)(L_{j,m}-L_{j,m-1})'(x)
\]
   It is shown in \cite{zhang:SIAM2012} that
\[
    m(L_m+L_{m-1})(s)=(s+1)(L_m-L_{m-1})'(s),\ \ s\in [-1,1],
\]
  where $L_m$ is the Legendre polynomial of degree $m$ in the
  interval $[-1,1]$.
  Then by a scaling from $[-1,1]$ to $[x_{j-\frac 12},x_{j+\frac
  12}]$, we obtain
\[
    \frac{\partial(u-P_h^-u)}{\partial x}(x,t)=(k+1)u_{j,k+1}\frac{(L_{j,k+1}+L_{j,k})(x)}{x-x_{j-\frac
    12}}+\sum_{p=k+2}^{\infty}u_{j,p}(L_{j,p}-L_{j,p-1})'(x).
\]
  Noticing that the first term of the above equation vanishes at the
  interior left Radau points $R_{j,l}, l\in\bZ_k$, we have
\[
   \frac{\partial(u-P_h^-u)}{\partial x}(R_{j,l},t)=\sum_{p=k+2}^{\infty}u_{j,p}(L_{j,p}-L_{j,p-1})'(R_{j,l}),\
   \ l\in\bZ_k.
\]
  Then the desired result \eqref{radau_1} follows by the standard
  approximation theory.
\end{proof}

  We are ready to show the superconvergence results of $u_h$ at the
  interior left Radau points.
\begin{theorem}\label{theo:2}
  Let $u\in W^{k+4,\infty}(\Omega)$  be the solution of
   \eqref{con_laws}, and $u_h$  the solution of \eqref{DG_scheme1}
    with initial value  $u_h(\cdot,0)$ be chosen such that \eqref{initial0} holds with
   $l=2$.
   Then for both the Dirichlet and periodic boundary  condition,
\begin{equation}\label{radau_2}
   \left|\frac{\partial(u-u_h)}{\partial x}(R_{j,l},t)\right|\lesssim (1+t)h^{k+1}\|u\|_{k+4,\infty},\ \  (j,l)\in\bZ_N \times \bZ_k.
\end{equation}
\end{theorem}
\begin{proof}
First, by Theorem \ref{theo:0} and the initial value chosen, we have
 \begin{equation*}\label{10}
   \|u_h-u_I^2\|_0 \lesssim (1+t)h^{k+3}\|u\|_{k+4,\infty}.
\end{equation*}
 Noticing $u_I^2=P_h^-u-w^2$ and
\begin{equation*}\label{11}
   \|w^2\|_{0,\infty,\tau_j}\lesssim h^{k+2}\|u\|_{k+3,\infty},
\end{equation*}
 we obtain
 \begin{equation}\label{00}
  \|u_h-P_h^-u\|_{0,\infty,\tau_j} \lesssim (1+t)h^{k+2}\|u\|_{k+4,\infty},\quad j\in\bZ_N.
\end{equation}
 Then by the inverse inequality,
\[
   |P_h^-u-u_h|_{1,\infty}\lesssim (1+t)h^{k+1}\|u\|_{k+4,\infty}.
\]
Consequently
\[
   \left|\frac{\partial(P_h^-u-u_h)}{\partial x}(R_{j,l},t)\right|\lesssim (1+t)h^{k+1}\|u\|_{k+4,\infty},\ \
   \forall j\in\bZ_N, l\in\bZ_k.
\]
Combining this  with the estimate \eqref{radau_1},  \eqref{radau_2}
follows.
\end{proof}

\subsection{Superconvergence at the right Radau
points}
  As we have mentioned in the introduction, one of the main theoretical results in \cite{Yang;Shu:SIAM2012}
    is the superconvergence rate $k+2$ for the function value error of $u-u_h$ at the right Radau points.
    A by-product of our analysis here is a different, and yet simpler way to establish this fact.

 Denote by $R^r_{j,l}, l\in \bZ_{k}$ the
   $k$ interior right Radau points in the interval $\tau_j, j\in\bZ_N$,
   namely, zeros of $L_{j,k+1}-L_{j,k}$ except the point $R_{j,0}^r=x_{j+\frac 12}$.  By the standard
   approximation theory
\[
    \left|(u-P_h^-u)(R_{j,l}^r,t)\right|=\left|\sum_{p=k+2}^{\infty}u_{j,p}(t)(L_{j,p}-L_{j,p-1})(R_{j,l}^r)\right|\lesssim
    h^{k+2}\|u\|_{k+2,\infty}.
\]
  On the other hand, if the initial value $u_h(\cdot,0)$ is chosen such that \eqref{initial0} holds with $l=2$, then
  \eqref{00} holds.
Consequently,
\begin{equation}\label{super_rrad}
    |(u-u_h)(R_{j,l}^r,t)|\lesssim (1+t)h^{k+2}\|u\|_{k+4,\infty},\ \
    \forall t\in (0,T].
\end{equation}
%\end{remark}

\begin{remark}  Similarly, if we choose the initial value
$u_h(\cdot,0)=u_I^l(\cdot,0)$ instead of letting $u_h(\cdot,0)$
satisfy \eqref{initial0} with $l=2$, the term $1+t$ in the estimates
\eqref{radau_2} and \eqref{super_rrad} can be improved to $t$.
 \end{remark}

To end  this section, we would like to demonstrate how to calculate
$u_I^l(x,0),1\le l\le k$
 only using the information of the initial value $u_0$.
Since $u_t+u_x=0$, we have for all integers $i\ge 1$,
\[
\frac{\partial^i}{\partial t^i}u(x,0)=(-1)^i u_0^{(i)}(x), \quad
\forall x\in\Omega.
\]
Therefore, by \eqref{coefficient}, for all $i\ge 1$, we have the
derivatives
\begin{equation}\label{derivm}
u^{(i)}_{j,k+1}(0)=(-1)^i\left(u_0^{(i)}(x^-_{j+\frac
12})-\frac{1}{h_j}\int_{\tau_j}
u_0^{(i)}(x)\sum_{m=0}^k(2m+1)L_{j,m}(x)dx\right).
\end{equation}
Now we divide the process into the following step :
\begin{itemize}
    \item[1.] In each element of $\tau_j$, calculate
    $G_i=u^{(i)}_{j,k+1}(0),i\in\bZ_l$ by \eqref{derivm}.
    \item[2.] Compute $F_i, i\in\bZ_l$ from \eqref{correcti} .
    \item[3.] Choose $w_i=\bar{h}_jF_iG_i$ and $w^l=\sum\limits_{i=1}^lw_i$.
    \item[4.] Figure out $u_I^l(x,0)=P^-_hu_0(x)-w^l(x,0)$.
\end{itemize}

\section{Numerical results}
   In this section, we present numerical examples to verify our
   theoretical findings.
   In our numerical experiments, we shall measure the
 maximum error and the average error at downwind points, the errors for the domain average and the cell average,
 the maximum derivative error at interior left Radau points and function value error at right Radau points, respectively.
  They are defined by

\begin{eqnarray*}
 &&e_{1}=\max_{j\in\bZ_N}\left|(u-u_h)(x^-_{j+\frac 12},T)\right|, \;\;
 e_{2}=\left(\frac 1N\sum_{j=1}^N\big(u-u_h\big)^2\big(x^-_{j+\frac 12},T\big)\right)^{\frac
 12},\\
&&e_{3}= \left|
\frac{1}{2\pi}\int_{0}^{2\pi}(u-u_h)(x,T)dx\right|,\;\;
 e_{4}=\max_{(j,l)\in\bZ_N\times\bZ_k}\left|\frac{\partial(u-u_h)}{\partial x}(R_{j,l},T)\right|\\
&&e_{5}=\max_{(j,l)\in\bZ_N\times\bZ_k}\left|(u-u_h)(R^r_{j,l},T)\right|,\;\;
 e_{6}=\left( \frac{1}{N}\sum_{j=1}^N\Big(\frac{1}{h_j}\int_{x_{j-\frac 12}}^{x_{j+\frac 12}}(u-u_h)(x,T)dx\Big)^2\right)^{\frac 12}.\\
\end{eqnarray*}

%%%%%%%%%%%%%%%%%%%%%%%%%%%%%%%%%%%%%%%%%%%%%%%%%%%%%%%%%%%%%

   To show
   the influence of the initial solution on the convergence rate, we
   also test four different methods for initial discretization in our
   experiments. There are

   Method 1: $u_h(x,0)=R_hu(x,0)$;

   Method 2: $u_h(x,0)=(P_h^-u)(x,0)$;

   Method 3: $u_{ht}(x,0)=(P_h^-u_t)(x,0),\; u_{h}(x^-_{j+\frac 12},0)=(P_h^-u)(x^-_{j+\frac 12},0) $;

   Method 4: $u_h(x,0)=u_I^k(x,0)$.\\
  Here $R_hu$ in Method 1 denotes the $L^2$ projection of $u$.
   Note that Method 4 is what we
   used in our superconvergence analysis, while Method 3 is a special way of initial
   discretization  proposed by Yang and Shu in \cite{Yang;Shu:SIAM2012}. In light of the frequent use of Methods
  1 and 2 for initial discretization of DG methods, we also test
   them in our experiments as comparison groups.

{\it Example} 1. We consider the following equation with the
 periodic boundary condition :
\begin{eqnarray*}
\begin{aligned}
   &u_t+u_x=0,\ \ &&(x,t)\in [0,2\pi]\times(0, {3\pi}/{4}], \\
   &u(x,0)=e^{\sin(x)},\ \  \\
   &u(0,t)=u(2\pi,t).
\end{aligned}
\end{eqnarray*}
   The exact solution to this problem is
\[
   u(x,t)=e^{\sin(x-t)}.
\]

  The problem is solved by the DG
  scheme \eqref{DG_scheme1} with $k=3,4$, respectively. Piecewise
  uniform meshes are used in our experiments, which are constructed by equally dividing
  each interval,  $[0,\frac{\pi}{2}]$ and $[\frac{\pi}{2},2\pi]$, into ${N}/{2}$ subintervals, $N=4,\ldots,512.$
  The ninth order strong-stability preserving (SSP) Runge-Kutta  method \cite{Gottlieb;Shu:SIAMSSP}
  with time step $\triangle t=0.05h_{min}, h_{min}={\pi}/{N}$ is used to reduce the time discretization error.

  Listed in Tables \ref{4}-\ref{7} are numerical data for errors
  $e_i,i=2,\ldots,6$ in cases $k=3,4$, with the initial solution obtained by Method 4.
  Depicted in Figure \ref{1_1} are corresponding
  error curves with log-log scale.

  We observe from Table \ref{7} and  Figure\ref{1_1} a convergence rate $k+1$ for $e_4$, $k+2$ for $e_5$, and $2k+1$ for
  $e_2$ and $e_3$, respectively. These results
   confirm our theoretical findings in Theorems \ref{theo:1}-\ref{theo:2},
   and \eqref{super_rrad} : The derivative error
  is superconvergent at all interior left Radau points and the
  function value error is superconvergent at all right Radau points,
 and the average error at downwind points is supercovergent as well
 as the error for the domain average, with a convergence rate $2k+1$.
 Moreover, we also observe numerically a $2k+1$ superconvergence
 rate for the cell average $e_6$.
Our numerical results demonstrate  that the superconvergence
 rates we proved in \eqref{averagenode}, \eqref{cell-average1}, \eqref{radau_2} and  \eqref{super_rrad} are optimal.

\begin{table}[htbp]\caption{ Errors $e_i, i=2,3,4$ in cases
$k=3,4$.}\label{4} \centering
\begin{threeparttable}
        \begin{tabular}{c c c c c c c c c c }
        \hline
        N & & $k=3$  & &  & & $k=4$ &  &   \\
     \cline{2-4} \cline{6-8}
       & $e_2$ & $e_3$&  $e_4$ && $e_2$  &$e_3$ & $e_4$   \\
        \hline\cline{1-9}
   4  &  2.45e-02 &  5.33e-03 &  1.10e-01 &&  5.48e-03 &  1.37e-03 &  5.53e-02 \\
   8  &  8.22e-04 &  4.64e-05 &  1.90e-02 &&  5.90e-05 &  2.15e-06 &  1.39e-03 \\
   16 &  1.04e-05 &  2.57e-07 &  1.08e-03 &&  2.12e-07 &  3.03e-09 &  9.37e-05 \\
   32 &  8.32e-08 &  1.83e-09 &  8.42e-05 &&  3.77e-10 &  4.96e-12 &  4.26e-06 \\
   64 &  6.73e-10 &  1.39e-11 &  5.34e-06 &&  7.61e-13 &  9.33e-15 &  1.47e-07 \\
  128 &  5.32e-12 &  1.08e-13 &  3.36e-07 &&  1.50e-15 &  1.81e-17 &  4.70e-09 \\
  256 &  4.17e-14 &  8.41e-16 &  2.10e-08 &&  2.94e-18 &  3.52e-20 &  1.47e-10 \\
  512 &  3.26e-16 &  6.57e-18 &  1.31e-09 &&  5.76e-21 &  6.87e-23 &  4.62e-12 \\
\hline
       \end{tabular}
 \end{threeparttable}
\end{table}
\begin{table}[htbp]\caption{Errors $e_5,e_6$ and corresponding
convergence rates in cases $k=3,4$.}\label{7} \centering
\begin{threeparttable}
        \begin{tabular}{c c c c c c c c c c c }
        \hline
        N & & $k=3$  & & & & & $k=4$ &  &   \\
     \cline{2-5} \cline{7-10}
       & $e_5$ & Rate&  $e_6$ & Rate && $e_5$  & Rate& $e_6$ & Rate  \\
        \hline\cline{1-10}
   4 &  5.39e-02  &   -- &  9.35e-03 &   --  & &1.34e-02 &    --   & 2.19e-03  &   --\\
   8 &  3.12e-03  & 4.11 &  1.60e-04 &  5.87 & & 2.15e-04 &   5.96 &  1.15e-05 &  7.57\\
   16&   8.99e-05&   5.12&   6.93e-06 &  4.53& &  4.31e-06 &  5.64 &  9.43e-08 &  6.93\\
   32&   2.68e-06 &  5.07 &  7.88e-08&   6.46& &  9.62e-08 &  5.49 &  3.48e-10 &  8.08\\
   64&   8.33e-08 &  5.01 &  6.66e-10 &  6.89 & & 1.65e-09 &  5.86 & 7.48e-13 &  8.86\\
   128 &  2.59e-09&   5.01&   5.32e-12&  6.97&  & 2.64e-11 &  5.97 &  1.50e-15 &  8.96\\
   256&   8.07e-11&   5.00 &  4.18e-14&  6.99 & & 4.14e-13 &  5.99 &  2.95e-18 &  9.00\\
   512 &  2.52e-12 &  5.00 &  3.27e-16&   7.00 &  &6.47e-15 & 6.00 &   5.77e-21 &  9.00\\
\hline
       \end{tabular}
 \end{threeparttable}
\end{table}

\begin{figure}[htbp]
%\begin{center}
\scalebox{0.5}{\includegraphics{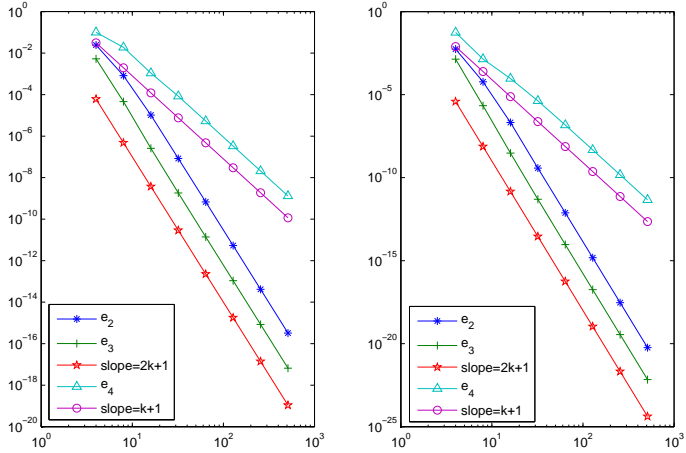}}
 \caption{left: $k=3$, right: $k=4$.}\label{1_1}
%\end{center}
\end{figure}

We also test the superconvergence for the maximum error at downwind
points by using the four different methods mentioned above for
initial discretization. We list in Tables \ref{5}-\ref{6}, the
approximation error $e_1$ and the corresponding convergence rate in
cases  $k=3,4$, respectively.  It seems that different choices of
the initial solution lead to different convergence rates. We observe
that when using Method 4, the convergence rate is of order $2k+1$,
$1/2$ order higher than the one given in \eqref{super_node2}. On the
other hand, Methods 1-3 do not result in the superconvergence rate
$2k+1$. Therefore, the way of initial discretization has influence
on the superconvergence rate at the downwind points.

\begin{table}[htbp]\caption{ $e_1$ and  corresponding convergence rates for different  initial
discretizations in case $k=3$.}\label{5} \centering
\begin{threeparttable}
        \begin{tabular}{c c c| c c |c c |c c  }
        \hline
        N & Method 1 &   & Method 2&  &Method 3 &  & Method 4 &   \\
     \cline{2-5} \cline{6-9}
       & $e_1$ & Rate &  $e_1$ & Rate& $e_1$  & Rate & $e_1$ & Rate   \\
        \hline\cline{1-9}
   4  &  4.09e-02 &   --  &  4.33e-02 &   --  &  4.37e-02  &  --  &  4.51e-02  &   -- \\
   8  &  2.09e-03 &  4.29 &  2.11e-03 &  4.36 &  1.95e-03  & 4.49 &  2.20e-03  & 4.36 \\
   16 &  5.63e-05 &  5.21 &  3.71e-05 &  5.83 &  3.39e-05  & 5.85 &  3.32e-05  & 6.05 \\
   32 &  1.40e-06 &  5.33 &  3.75e-07 &  6.63 &  3.67e-07  & 6.53 &  3.10e-07  & 6.74 \\
   64 &  5.02e-08 &  4.80 &  3.57e-09 &  6.72 &  3.46e-09  & 6.73 &  2.53e-09  & 6.94 \\
  128 &  1.97e-09 &  4.67 &  6.01e-11 &  5.89 &  3.32e-11  & 6.71 &  2.00e-11  & 6.98 \\
  256 &  8.43e-11 &  4.55 &  1.06e-12 &  5.82 &  3.40e-13  & 6.61 &  1.57e-13  & 6.99 \\
  512 &  3.73e-12 &  4.50 &  2.49e-14 &  5.41 &  3.77e-15  & 6.49 &  1.23e-15  & 7.00 \\
\hline
       \end{tabular}
 \end{threeparttable}
\end{table}

\begin{table}[htbp]\caption{ $e_1$ and corresponding convergence rates for different  initial
discretizations in case $k=4$.}\label{6} \centering
\begin{threeparttable}
        \begin{tabular}{c c c| c c |c c |c c  }
        \hline
        N & Method 1 &   & Method 2&  &Method 3 &  & Method 4 &   \\
     \cline{2-5} \cline{6-9}
       & $e_1$ & Rate &  $e_1$ & Rate& $e_1$  & Rate & $e_1$ & Rate   \\
        \hline\cline{1-9}
   4  & 1.05e-02  &  --  & 1.06e-02 &    -- &  1.03e-02 &  --  &  1.09e-02 &  -- \\
   8  & 2.12e-04  & 5.63 & 1.61e-04 &  6.04 &  1.61e-04 & 6.00 &  1.60e-04 & 6.09 \\
   16 & 2.27e-06  & 6.54 & 8.23e-07 &  7.61 &  7.63e-07 & 7.72 &  6.19e-07 & 8.02 \\
   32 & 8.71e-08  & 4.71 & 5.25e-09 &  7.29 &  2.21e-09 & 8.43 &  1.44e-09 & 8.75 \\
   64 & 1.98e-09  & 5.46 & 8.45e-11 &  5.96 &  9.46e-12 & 7.87 &  2.94e-12 & 8.94 \\
  128 & 1.67e-11  & 6.89 & 1.04e-12 &  6.35 &  2.76e-14 & 8.42 &  5.82e-15 & 8.98 \\
  256 & 6.36e-13  & 4.71 & 1.17e-14 &  6.47 &  1.72e-16 & 7.32 &  1.14e-17 & 9.00 \\
  512 & 1.02e-14  & 5.96 & 1.25e-16 &  6.55 &  9.06e-19 & 7.57 &  2.23e-20 & 9.00 \\

\hline
       \end{tabular}
 \end{threeparttable}
\end{table}

%-----------------------------------------------------------------------------------------------------------

 {\it Example} 2. We consider the following problem with the Dirichlet boundary
 condition :
\begin{eqnarray*}
\begin{aligned}
   &u_t+u_x=0,\ \ &&(x,t)\in [0,2\pi]\times(0,\pi], \\
   &u(x,0)=\sin(x), \\
   &u(0,t)=-\sin(t).
\end{aligned}
\end{eqnarray*}
     The exact solution to this problem is
\[
   u(x,t)=\sin(x-t).
\]

   We construct our meshes by dividing the interval
  $[0,2\pi]$ into $N$ subintervals, $N=2,\ldots, 64$, and solve this problem  by the DG
  scheme \eqref{DG_scheme1} with polynomial
  degree $k=3,4$, respectively. To  diminish the time discretization error, we use the fourth order
  Runge-Kutta method with time step $\triangle t=T/n$ for $n=10N^2$
  in $k=3$ and $n=5N^3$ in $k=4$.

Numerical data are demonstrated in Tables \ref{1}-\ref{8}, and
corresponding error curves are depicted in Figure \ref{1_2} on the
log-log scale with the initial solution obtained by Method 4. Again,
we observe a convergence rate $k+1$ for $e_4$, $k+2$ for $e_5$ and
$2k+1$ for $e_2,e_3$ and $e_6$, respectively. These results
 verify our theoretical findings in Theorems \ref{theo:1}-\ref{theo:2}
 and \eqref{super_rrad}.   Note that the superconvergence rate
 $2k+1$ for the domain average is $ 1/2$ order
 higher than the one given in \eqref{cell-average}.

\begin{table}[htbp]\caption{ Errors $e_i, i=2,3,4$ in cases
$k=3,4$.}\label{1} \centering
\begin{threeparttable}
        \begin{tabular}{c c c c c c c c c c }
        \hline
        N & & $k=3$  & &  & & $k=4$ &  &   \\
     \cline{2-4} \cline{6-8}
       & $e_2$ & $e_3$&  $e_4$ && $e_2$  &$e_3$ & $e_4$   \\
        \hline\cline{1-9}
   2  &  1.83e-03 &  8.64e-04 & 1.01e-02 && 5.00e-05 &  2.77e-05 &  7.08e-03 \\
   4  &  2.68e-05 &  8.04e-06 & 2.14e-03 && 2.11e-07 &  6.20e-08 & 1.85e-04  \\
   8  &  2.22e-07 &  6.56e-08 & 1.66e-04 && 4.29e-10 &  1.25e-10 &  7.24e-06 \\
   16 &  1.78e-09 &  5.14e-10 & 1.09e-05 && 8.56e-13 &  2.45e-13 &  2.38e-07 \\
   32 &  1.41e-11 &  4.01e-12 & 6.90e-07 && 1.68e-15 &  4.77e-16 &  7.55e-09 \\
   64 &  1.10e-13 &  3.12e-14 & 4.31e-08 && 3.29e-18 &  9.29e-19 &  2.36e-10 \\
\hline
       \end{tabular}
 \end{threeparttable}
\end{table}

\begin{table}[htbp]\caption{Errors $e_5,e_6$ and corresponding
convergence rates in cases $k=3,4$.}\label{8} \centering
\begin{threeparttable}
        \begin{tabular}{c c c c c c c c c c c }
        \hline
        N & & $k=3$  & & & & & $k=4$ &  &   \\
     \cline{2-5} \cline{7-10}
       & $e_5$ & Rate&  $e_6$ & Rate && $e_5$  & Rate& $e_6$ & Rate  \\
        \hline\cline{1-10}
  2 &  7.60e-03 &   --  & 2.41e-03    &    --  & & 1.75e-03 &        --  &  7.97e-05  &          --\\
   4&   3.96e-04 &   4.26  &  2.61e-05 &   6.53 & & 2.29e-05&    6.26&    2.05e-07&    8.60\\
   8 &   1.38e-05 &   4.85&    2.41e-07 &   6.76& &   4.36e-07 &   5.71 &   4.64e-10&    8.79\\
   16&    4.44e-07&    4.95&    1.98e-09&    6.93 & &  7.14e-09&    5.93&    9.43e-13&    8.94\\
   32 &   1.40e-08 &   4.99 &   1.57e-11&    6.98& &   1.13e-10 &   5.98 &   1.86e-15&    8.98\\
   64 &   4.39e-10&    5.00 &   1.23e-13&    6.99 & &  1.77e-12 &   6.00 &   3.66e-18 &   8.99\\
\hline
       \end{tabular}
 \end{threeparttable}
\end{table}

\begin{figure}[htbp]
%\begin{center}
\scalebox{0.5}{\includegraphics{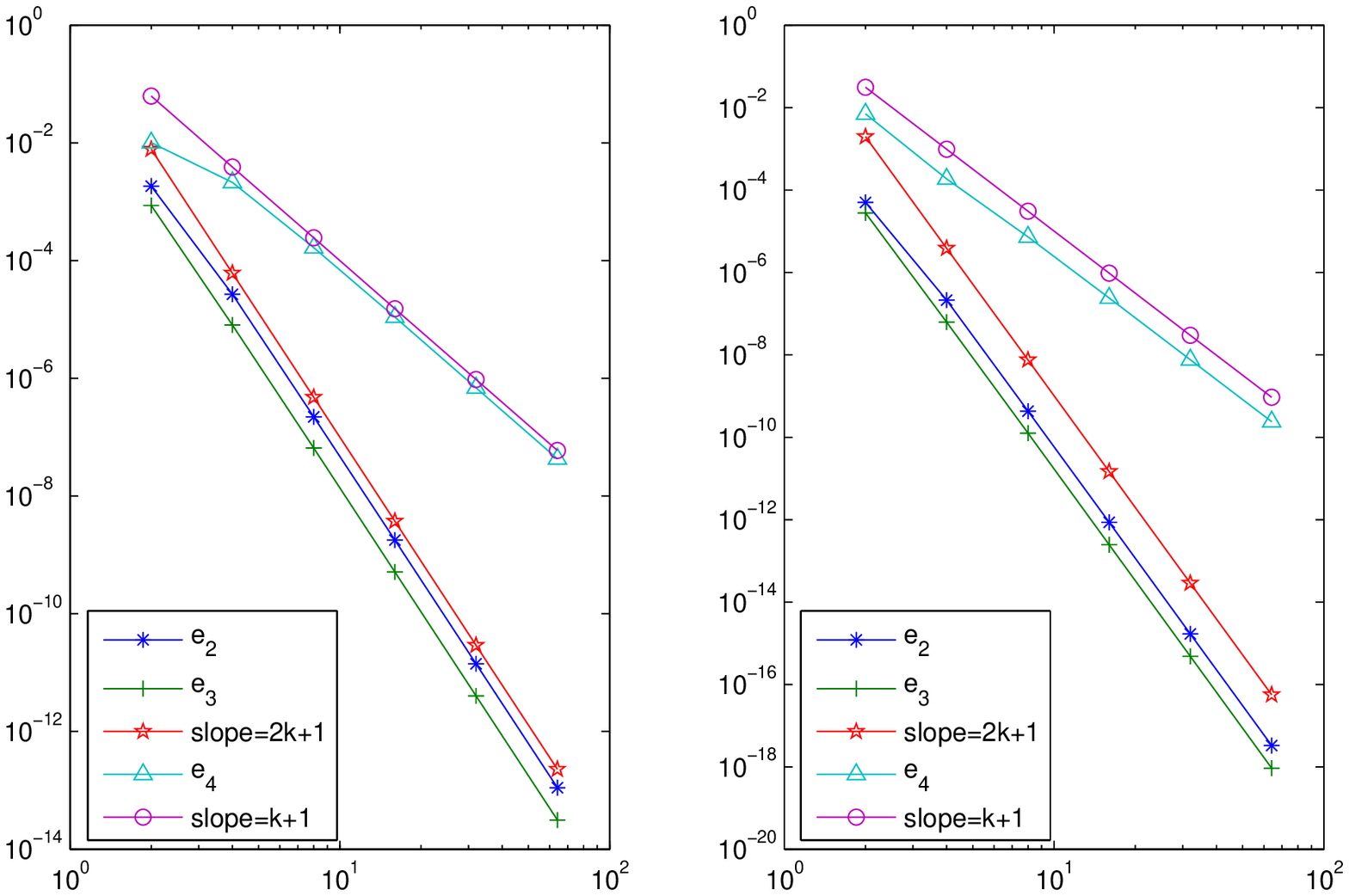}}
 \caption{left: $k=3$, right: $k=4$.}\label{1_2}
%\end{center}
\end{figure}

As in Example 1, we also test  convergence rates at
 the downwind points under aforementioned four different initial discretization methods.
 Tables \ref{2}-\ref{3}  demonstrate corresponding errors and convergence rates, from which, we observe similar results as in the
 periodic boundary condition :  the convergence rate of $e_1$ is $2k+1$ for Method 4 while not for Methods 1-2.
 As for Method 3, it seems that the superconvergence rate is $2k+1$  for $k=3$. However, it is not valid  for $k=4$.

\begin{table}[htbp]\caption{ $e_1$ and corresponding convergence rates for different initial
discretizations in case $k=3$.}\label{2} \centering
\begin{threeparttable}
        \begin{tabular}{c c c| c c |c c |c c  }
        \hline
        N & Method 1 &   & Method 2&  &Method 3 &  & Method 4 &   \\
     \cline{2-5} \cline{6-9}
       & $e_1$ & Rate &  $e_1$ & Rate& $e_1$  & Rate & $e_1$ & Rate   \\
        \hline\cline{1-9}
   2   & 8.23e-03  &   --   & 4.63e-03  & --   & 4.88e-03   &   --  &  1.94e-03   &  -- \\
   4   & 2.88e-04  & 4.83   & 2.23e-05  & 7.70 & 5.93e-05   & 6.36  &  4.61e-05   & 5.39 \\
   8   & 1.26e-05  & 4.51   & 1.11e-06  & 4.33 & 5.21e-07   & 6.83  &  3.92e-07   & 6.88 \\
   16  & 1.81e-07  &  6.15  & 1.74e-08  & 6.00 & 3.54e-09   & 7.20  &  3.16e-09   & 6.96  \\
   32  & 6.10e-10  &  8.23  & 2.94e-10  & 5.89 & 2.96e-11   & 6.90  &  2.49e-11   & 6.99  \\
   64  & 1.39e-11  &  5.44  & 4.66e-12  & 5.98 & 2.32e-13   & 7.00  &  1.95e-13   & 7.00   \\
\hline
       \end{tabular}
 \end{threeparttable}
\end{table}

\begin{table}[htbp]\caption{ $e_1$ and corresponding convergence rates for different initial
discretizations in case $k=4$.}\label{3}
 \centering
\begin{threeparttable}
        \begin{tabular}{c c c| c c |c c |c c  }
        \hline
        N & Method 1 &   & Method 2&  &Method 3 &  & Method 4 &   \\
     \cline{2-5} \cline{6-9}
       & $e_1$ & Rate &  $e_1$ & Rate& $e_1$  & Rate & $e_1$ & Rate   \\
        \hline\cline{1-9}
  2 &  1.43e-04 &   --  &  8.26e-05  &  --  & 1.09e-04  &  --   & 5.25e-05 &  --\\
   4 &  2.69e-05 &  2.41 &  2.02e-06  & 5.36 & 1.00e-06  &  6.77 & 3.66e-07 &  7.16\\
   8 &  7.85e-07 &  5.10 &  1.25e-08  & 7.33 & 1.10e-08  &  6.51 & 7.60e-10 &  8.91\\
  16 &  2.02e-08 &  5.28 &  9.26e-11  & 7.08 & 7.71e-11  &  7.16 & 1.51e-12 &  8.97\\
  32 &  3.81e-10 &  5.73 &  5.13e-12  & 4.18 & 3.72e-13  &  7.70 & 2.96e-15 &  8.99\\
  64 &  4.78e-12 &  6.31 &  7.23e-14  & 6.15 & 1.52e-15  &  7.94 & 5.80e-18 &  9.00\\
\hline
       \end{tabular}
 \end{threeparttable}
\end{table}

\vskip.1in

\section {Conclusion}
 In this work, we have studied superconvergence
properties of the DG method for linear hyperbolic conservation laws
under the one spacial dimension setting. Our main theoretical result
is the proof of $2k+1$-superconvergence rate at the downwind points
in an average sense (Theorem \ref{theo:1}, equations
\eqref{averagenode}   and \eqref{averagenode1}) as well as for the
domain average (Theorem \ref{coro:1}, equation
\eqref{cell-average1}), and thereby settle a long standing
theoretical conjecture. An unexpected discovery is that in order to
achieve the $2k+1$ rate, a proper implementation of the initial
solution based on the correction procedure introduced in this paper
is crucial for $k>3$. This observation is supported by a numerical
comparison with traditional implementations of the initial solution.
Indeed, only Method 4 (based on our correction scheme) can achieve
$2k+1$ rate for $k=4$.

As a by-product, we also proved, for the first time, a point-wise
derivative superconvergence rate $k+1$ at all left Radau points
(Theorem \ref{theo:2}, equation \eqref{radau_2}). At this point, our
proof for the point-wise superconvergence rate $2k+1/2$ at the
downwind points (Theorem \ref{theo:1}, equations \eqref{super_node2}
and \eqref{super_node3}) is still sub-optimal (comparing with the
numerical rate $2k+1$). In addition, the proof of $2k+1$ rate for
the cell average remains open. Our other on-going works include
convection-diffusion equations as discussed in
\cite{Chen;Shu:SIAM2010} and higher dimensional conservation laws.

%%%%%%%%%%%%%%%%%%%%%%%%%%%%%%%%%%%%%%%%%%%%%%%%%%%%%%%%%%%%%%%%%%%%%%%%%%%%%%%%%%%%%%%%%%%%%%%%%%%%%%%%%%%%%%%%%%%%%%%%%%%%%%%%%%%%%%

\end{document}